\documentclass{amsart}
\usepackage{graphicx}
\usepackage{amssymb}

\newtheorem{thm}{Theorem}[section]
\newtheorem{cor}[thm]{Corollary}
\newtheorem{lem}[thm]{Lemma}
\newtheorem{prop}[thm]{Proposition}
\theoremstyle{definition}
\newtheorem{defn}[thm]{Definition}
\theoremstyle{remark}
\newtheorem{rem}[thm]{Remark}
\theoremstyle{conjecture}
\newtheorem{conj}[thm]{Conjecture}
\theoremstyle{example}
\newtheorem{ex}[thm]{Example}

\begin{document}

\title[Y-shaped singularities]{Solutions to the Monge-Amp\`{e}re equation with polyhedral and Y-shaped singularities}
\author{Connor Mooney}
\address{Department of Mathematics, UC Irvine}
\email{\tt mooneycr@math.uci.edu}

\begin{abstract}
We construct convex functions on $\mathbb{R}^3$ and $\mathbb{R}^4$ that are smooth solutions to the Monge-Amp\`{e}re equation 
$$\det D^2u = 1$$
away from compact one-dimensional singular sets, which can be Y-shaped or form the edges
of a convex polytope. The examples solve the equation in the Alexandrov sense away from finitely many points. 
Our approach is based on solving an obstacle problem where the graph of the obstacle is a convex polytope.
\end{abstract}
\maketitle
\section{Introduction}
The problem of constructing singular Monge-Amp\`{e}re metrics has attracted recent attention due to its connections with mirror symmetry (\cite{Li}, \cite{Lo}, \cite{LYZ}, \cite{CL2}, \cite{JX}). 
By a singular Monge-Amp\`{e}re metric we mean the Hessian of a convex function $u$ on $\mathbb{R}^n$ that is a smooth solution to the Monge-Amp\`{e}re equation
\begin{equation}\label{MA}
\det D^2u = 1
\end{equation}
away from a small singular set $\Gamma$, where $\Delta u$ blows up. Of particular interest seems to be the case that $\Gamma$ is a trivalent graph and $n = 3$. 

The purpose of this paper is to develop a robust method for constructing such examples. In particular, we show:
\begin{thm}\label{Main}
Let $\Omega \subset \mathbb{R}^n$ be a compact convex polytope, and let $\Gamma_k$ denote its $k$-skeleton. Assume further that $n = 3$ or $n = 4$. Then there exists a convex function $u: \mathbb{R}^n \rightarrow \mathbb{R}$ such that $\Gamma_1 \subset \{u = 0\},\, u \in C^{\infty}(\mathbb{R}^n \backslash \Gamma_1)$, and
\begin{equation}\label{Deltas}
\det D^2u = 1 + \sum_{q \in \Gamma_0} a_q\delta_q
\end{equation}
in the Alexandrov sense, for some coefficients $a_q > 0$.
\end{thm}

\noindent By $k$-skeleton we mean the collection of faces of dimension at most $k$. We also construct examples where the singular set is a Y-shape that may lie in a plane (in which case it is not contained in a level set of $u$) or not. In \cite{JX} the authors studied equations of the form (\ref{Deltas}) in $\mathbb{R}^n$, where the right hand side is the sum of a constant with some Dirac masses. They showed when $n \leq 4$ that global solutions are smooth away from the collection of line segments that connect pairs of masses, but it remained open whether the solutions could in fact be singular on these segments. Theorem \ref{Main} answers this question positively.

\begin{rem}
Such examples are not possible in dimension $n = 2$ because in that case, solutions to $\det D^2u \geq 1$ are strictly convex (\cite{A}, see \cite{M1} for a proof).
\end{rem}

\begin{rem}
The polytope $\Omega$ from Theorem \ref{Main} is allowed to be degenerate, e.g. the convex hull of $n$ or fewer points.
\end{rem}
 
Our approach to Theorem \ref{Main} is based on solving an obstacle problem for the Monge-Amp\`{e}re equation, with an obstacle whose graph is a convex polytope. To obtain the examples we take the Legendre transform of solutions to the obstacle problem. By playing with the choice of obstacle, one can in fact construct a variety of singular examples in $\mathbb{R}^3$ and $\mathbb{R}^4$ where the singular set is a graph that need not be the set of edges of a convex polytope or a trivalent graph (see the discussion at the beginning of Section \ref{YSing}). In \cite{S} Savin studies an obstacle problem for the Monge-Amp\`{e}re equation with a linear obstacle, and the existence and regularity theory developed in that paper are important in our constructions. 

Most of the steps in the proof of Theorem \ref{Main} work in any dimension. We specialize to dimension $n \leq 4$ to prove a certain qualitative regularity result about solutions to the obstacle problem (see Lemma \ref{LowDRegProp}). Provided the analogue of that result holds in higher dimensions, we can generalize Theorem \ref{Main} to higher dimensions (see Proposition \ref{RegProp}). We state the generalization here as a conjecture:

\begin{conj}\label{nDMain}
Let $\Omega \subset \mathbb{R}^n$ be a compact convex polytope, and let $\Gamma_k$ denote its $k$-skeleton. Then there exists a convex function $u : \mathbb{R}^n \rightarrow \mathbb{R}$ such that
$$\Gamma_{\left\lceil \frac{n}{2} - 1\right\rceil} \subset \{u = 0\}, \quad u \in C^{\infty}\left(\mathbb{R}^n \backslash\Gamma_{\left\lceil \frac{n}{2} - 1\right\rceil}\right), \quad \text{and} \quad \det D^2u = 1 + \sum_{q \in \Gamma_0} a_q\delta_q$$
for some coefficients $a_q > 0$.
\end{conj}

\noindent Here $\lceil t \rceil$ denotes the smallest integer greater than or equal to $t$. We can verify Conjecture \ref{nDMain} when $n \leq 4$, and in some simple cases when $n \geq 5$. It is not clear whether it is reasonable to expect that it holds in general (see Remark \ref{5DCase}).

\begin{rem}
The subgradient maps of the Legendre transforms of the examples from Theorem \ref{Main} can be viewed as optimal transport maps for the quadratic cost. They push forward the uniform measure to the uniform measure plus a sum of Dirac masses, and they transport convex sets with nonempty interior to the points where the masses are centered.
\end{rem}

\begin{rem}
At infinity, the solutions from Theorem \ref{Main} have the asymptotic behavior
$$u(x) = \text{const.} + \frac{1}{2}|x|^2 + O(|x|^{2-n}).$$ 
By the results in \cite{CL} and \cite{JX} about solutions to $\det D^2u = 1$ in exterior domains, these are the unique solutions with this asymptotic behavior.
In this paper we are concerned with the local behavior of solutions.
\end{rem}

The paper is organized as follows. In Section \ref{Preliminaries} we prove some preliminary results. In particular, we solve an obstacle problem for the Monge-Amp\`{e}re equation with general convex obstacle, construct a family of barriers, prove a result about propagation of singularities, and state versions of a few regularity results from \cite{S}. In Section \ref{PolyObstacle} we construct, in any dimension, global solutions to an obstacle problem where the graph of the obstacle is a convex polytope. In Section \ref{LowD} we prove that if the global solutions to the obstacle problem satisfy a certain qualitative regularity condition, then their Legendre transforms settle Conjecture \ref{nDMain}. We then prove that this condition is satisfied when $n \leq 4$ to obtain Theorem \ref{Main}. Finally, in Section \ref{YSing} we explain how to modify the approach to construct examples where the singular set is a Y-shape.

\section*{Acknowledgments}
The author is grateful to Tianling Jin and Jingang Xiong for asking the question that motivated this research, and to Richard Schoen for a helpful discussion. The research was supported by NSF grant DMS-1854788.


\section{Preliminaries}\label{Preliminaries}

In this section we recall the notion of Monge-Amp\`{e}re measure, solve an obstacle problem for the Monge-Amp\`{e}re equation, build a family of barriers, prove a result about the propagation of singularities, and recall some regularity results from \cite{S}.

\subsection{Monge-Amp\`{e}re Measure}
To any convex function $v$ on a domain $\Omega \subset \mathbb{R}^n$ we associate a Borel measure $Mv$ on $\Omega$, called the Monge-Amp\`{e}re measure of $v$, that satisfies
$$Mv(E) = |\partial v(E)|$$
for any Borel set $E \subset \Omega$. Here $\partial v$ denotes the subgradient of $v$. When $v \in C^2$ we have $Mv = \det D^2v \,dx.$
Given a Borel measure $\mu$ on $\Omega$, we say that $v$ is an Alexandrov solution to the Monge-Amp\`{e}re equation 
$$\det D^2v = \mu$$ 
if 
$$Mv = \mu.$$
Alexandrov solutions are closed under uniform convergence: if convex functions $v_k$ converge locally uniformly in $\Omega$ to
$v$, then the Monge-Amp\`{e}re measures $Mv_k$ converge weakly to $Mv$. 
For proofs of these results see \cite{Gut}.


\subsection{Obstacle Problem}\label{ObstacleProblem}
Let $\Omega \subset \mathbb{R}^n$ be a bounded strictly convex domain, $\psi$ a convex function from $\mathbb{R}^n$ to $\mathbb{R}$, $\mu$ a finite Borel measure on $\Omega$,
and $\varphi \in C(\partial\Omega)$ with $\varphi > \psi$ on $\partial \Omega$. We define
$$\mathcal{F} := \left\{v : v \in C\left(\overline{\Omega}\right) \text{ convex},\, v \geq \psi \text{ in } \Omega,\, v|_{\partial \Omega} = \varphi,\, Mv \leq \mu\right\}.$$
We note that the convex envelope $\Phi$ of $\varphi$ is in $\mathcal{F}$ and has vanishing Monge-Amp\`{e}re measure (see \cite{Gut}). 

The main result of this subsection is the solvability of an obstacle problem:

\begin{prop}\label{OP}
The function 
$$u := \inf_{\mathcal{F}} v$$ 
is in $\mathcal{F}$, and 
$$Mu = \mu \text{ in } \{u > \psi\} \cap \Omega.$$
\end{prop}

\noindent The proof of Proposition \ref{OP} follows the same lines as that of Proposition 1.1 in \cite{S}, where the special case $\psi = 0$ and $\varphi = \text{const.}$ is considered. The key points are the equicontinuity of $\mathcal{F}$ (a consequence of the Alexandrov maximum principle and the continuity of $\Phi$ up to the boundary), the closedness of $\mathcal{F}$ under taking the convex envelope of the minimum of two functions in $\mathcal{F}$ (see \cite{S}),
the closedness of $\mathcal{F}$ under uniform convergence (a consequence of the weak convergence of Monge-Amp\`{e}re measures), and the solvability of the Dirichlet problem for Alexandrov solutions in strictly convex domains (see \cite{Gut}).

\begin{rem}\label{ModifiedPerron}
We can also write $u$ as the infimum of functions in 
$$\mathcal{\tilde{F}} := \left\{\tilde{v} : \tilde{v} \in C\left(\overline{\Omega}\right) \text{ convex},\, \tilde{v} \geq \psi \text{ in } \Omega,\, \tilde{v}|_{\partial \Omega} \geq \varphi,\, M\tilde{v} \leq \mu\right\}.$$
Indeed, for any function $\tilde{v} \in \mathcal{\tilde{F}}$ there is a function $v \in \mathcal{F}$ such that $v \leq \tilde{v}$, given by the convex envelope of $\min\{\Phi,\,\tilde{v}\}$.
\end{rem}


\subsection{Barriers}\label{Barriers}
We denote points in $\mathbb{R}^n$ by $(x,\,y)$ with $x \in \mathbb{R}^{n-k}$ and $y \in \mathbb{R}^k$. For $n \geq 3$ and $1 \leq k < \frac{n}{2}$, we let
$$\gamma := 2\frac{n-k}{n-2k}, \quad r := |x|, \quad t := |y|, \quad \text{and} \quad s := r^{-\gamma}t.$$ 
Then the function $w_{n,\,k}$ defined on $\mathbb{R}^n$ by
\begin{equation}\label{Super}
\left(\frac{\gamma}{2}\right)^{1-\frac{k}{n}}(\gamma-1)^{\frac{1}{n}}\, w_{n,\,k}(x,\,y) := \begin{cases}
\frac{1}{2}(r^{\gamma}  + r^{-\gamma}t^2), \quad t \leq r^{\gamma} \\
t, \quad t > r^{\gamma}
\end{cases}
\end{equation}
satisfies
$$\det D^2w_{n,\,k} = (1-s^2)^{n-k}\chi_{\{t < r^{\gamma}\}}$$
in the Alexandrov sense. We omit the calculation, which is straightforward using coordinates that are polar in $x$ and $y$. In particular,
\begin{equation}\label{SuperIneq}
\det D^2w_{n,\,k} \leq 1.
\end{equation}
We also note that $w_{n,\,k}(0) = 0$ and that
\begin{equation}\label{LipschitzSing}
w_{n,\,k}(x,\,y) \geq c(n,\,k)|y|,
\end{equation}
so $w_{n,\,k}$ has a Lipschitz singularity at the origin.

\begin{rem}
The functions $w_{n,\,k}$ resemble the Legendre transforms of the Pogorelov example $|x'|^{2-\frac{2}{n}}(1+x_n^2)$ in the case $k = 1$ (here $(x',\,x_n) \in \mathbb{R}^n$) and its generalizations \cite{C}, which are nonnegative and have Monge-Amp\`{e}re measure bounded between positive constants near the origin, but vanish on convex sets of dimension $k$.
\end{rem}

\noindent We extend the definition of $w_{n,\, k}$ to $n \geq 1$ and $k = 0$ by taking
$$w_{n,\,0}(x) := \frac{1}{2}|x|^2.$$
Finally, for $n \geq 1$ we let
\begin{equation}\label{FlatObstacle}
W_n(x) := \int_0^{|x|} \left(s^n - 1\right)_{+}^{\frac{1}{n}}\,ds,
\end{equation}
which vanishes in $B_1$ and solves
\begin{equation}\label{FlatSuper}
\det D^2W_n = \chi_{\{|x| > 1\}}
\end{equation}
in the Alexandrov sense. It also satisfies
\begin{equation}\label{GrowthRadial}
W_n(x) - \frac{1}{2}|x|^2 = \begin{cases}
O(|x|), \quad n = 1 \\
O(|\log|x||), \quad n = 2 \\
c(n) + O(|x|^{2-n}), \quad n \geq 3
\end{cases}
\end{equation}
for some constants $c(n) < 0$.

\subsection{A Propagation Result}\label{PropagationResult}
In this subsection we prove a propagation result that complements the family of barriers $\{w_{n,\,k}\}$. Again we denote points in $\mathbb{R}^n$ by $(x,\,y)$ with $x \in \mathbb{R}^{n-k}$ and $y \in \mathbb{R}^{k}$.

\begin{prop}\label{Propagation}
Assume $\det D^2u \leq \Lambda < \infty$ in $B_1 \subset \mathbb{R}^n$. If $k \geq \frac{n}{2}$, $u(0) = 0$ and $u \geq |y|$, then $\{u = 0\}$ has no extremal points in $\{y = 0\} \cap B_1$. 
\end{prop}

\noindent The case $n = 2,\, k = 1$ is a classical result of Alexandrov (\cite{A}, see also \cite{M2}). Proposition \ref{Propagation} can be viewed as a generalization of this result to higher dimensions.
The proof relies on the following volume estimate (see e.g. Lemma $2.5$ from \cite{M1}):

\begin{lem}\label{VolumeGrowth}
Assume that $\det D^2u \leq 1$ on a bounded convex domain $\Omega$ and that $u|_{\partial \Omega} = 0$. Then
$$|\Omega| \geq c(n)\left|\min_{\Omega} u\right|^{\frac{n}{2}}.$$
\end{lem}

\noindent We proceed with the proof of the propagation result. Below, a Lipschitz rescaling refers to a rescaling of the form $u \rightarrow \lambda^{-1}u(\lambda x)$ with $\lambda > 0$.

\begin{proof}[{\bf Proof of Proposition \ref{Propagation}:}]
Assume by way of contradiction that $\{u = 0\}$ has an extremal point in $\{y = 0\} \cap B_1$. Then after an affine transformation in $x$ and a Lipschitz rescaling we may assume that the domain of definition for $u$ is still $B_1$, that $\Lambda = 1$, that $u(0) = 0$, that
$$\{u = 0\} \subset \{y = 0\} \cap \{x_1 \leq 0\},$$
and that 
$$\{u = 0\} \cap \{x_1 \geq -\delta_0\} \subset \subset B_1$$ 
for some $\delta_0 > 0$. Fix $\delta < \delta_0$ and let 
$$l_h := h\,(x_1 + \delta).$$
Then for $h$ small we have that
$\{u < l_h\} \subset \subset B_1$ and furthermore that
$$\{u < l_h\} \subset \{|x_1| < \delta\} \cap \{|y| < l_h\}.$$
It follows that
$$|\{u - l_h < 0\}| \leq C(n)\,\delta\,(\delta h)^k.$$
Applying Lemma \ref{VolumeGrowth} to the function $u - l_h$ and using that $|(u-l_h)(0)| = \delta h$ we conclude that
$$(\delta h)^{\frac{n}{2}} \leq C(n)\,\delta (\delta h)^k,$$
which is not possible when $k \geq \frac{n}{2}$ and $\delta$ is chosen small.
\end{proof}

\noindent For our purposes, the following corollary of Proposition \ref{Propagation} will be useful:

\begin{cor}\label{PropagationCor}
Assume that $\det D^2u \leq 1$ on $\mathbb{R}^n$, and that $\det D^2u$ is not identically zero. Then $\text{dim}(\partial u(p)) < \frac{n}{2}$ for all $p \in \mathbb{R}^n$.
\end{cor}
\begin{proof}
Assume by way of contradiction that $\partial u(0)$ has dimension $k \geq \frac{n}{2}$, and let $(x,\,y) \in \mathbb{R}^n$ with $x \in \mathbb{R}^{n-k}$ and $y \in \mathbb{R}^k$. After subtracting a linear function, rotating, and multiplying by a constant, we may assume that $\det D^2u \leq \Lambda < \infty$ and that $B_1 \cap \{x = 0\} \subset \partial u(0)$. Then $u \geq u(0) + |y|$. By Proposition \ref{Propagation}, the set $\{u = u(0)\}$ contains a line. By convexity, $u$ is invariant under translations along this line, hence $\det D^2u \equiv 0$.
\end{proof}

\subsection{Regularity Results}\label{Regularity}
To conclude the section we state some regularity results for solutions to the obstacle problem with linear obstacle. These can be viewed as localized versions of results from \cite{S}. The first result is a consequence of Lemma $3.3$ from \cite{S}:
\begin{prop}\label{StrictConvexity}
Assume that $\det D^2u = \chi_{\{u > 0\}}$ in a bounded convex domain $\Omega \subset \mathbb{R}^n$, and that $u \geq 0$. Let $S$ be a supporting hyperplane to $\{u = 0\}$ at a point in $\Omega$. Then $S \cap \{u = 0\}$
is either a single point, or it has no extremal points in $\Omega$.
\end{prop}
\noindent In particular, if $x \in \partial\{u = 0\} \cap \Omega$, then there are two possibilities: either every supporting hyperplane to $\{u = 0\}$ at $x$ intersects $\{u = 0\}$ only at $x$, or else $x$ lies in the interior of 
a segment in $\partial\{u = 0\}$ with an endpoint on $\partial \Omega$. By combining this result with the proof of Proposition $2.8$ from \cite{S} we obtain:
\begin{prop}\label{ExtremalC1}
Let $u$ be as in Proposition \ref{StrictConvexity}. Then $\partial u(p) = \{0\}$ at every extremal point of $p$ of $\{u = 0\}$ in $\Omega$.
\end{prop}

\begin{rem}
These regularity results still hold when we replace the right hand side by $f\,\chi_{\{u > 0\}}$, for any function $f$ that is bounded above and below by positive constants.
\end{rem}

\section{Global Obstacle Problem}\label{PolyObstacle}
In this section we let $\Omega \subset \mathbb{R}^n$ be a compact convex polytope of dimension $d$, and we denote by $\Gamma_k$ its $k$-skeleton (that is, the union of its faces of dimension at most $k$). With the convention that $\Gamma_k = \emptyset$ when $k < 0$, we let
$$S_k := \Gamma_k \backslash \Gamma_{k-1}$$
be the union of the interiors of its $k$-dimensional faces. Finally, we assume that the $d$-dimensional interior of $\Omega$ contains the origin, and when we write $\mathbb{R}^n = \mathbb{R}^{n-d} \times \mathbb{R}^{d}$ that $\Omega \subset \{0\} \times \mathbb{R}^d$.

We define the convex function $P$ on $\mathbb{R}^n$ by
\begin{equation}\label{PDef}
P(x) := \begin{cases}
            0, \quad x \in \Omega \\
            +\infty, \quad x \in \mathbb{R}^n \backslash \Omega,
\end{cases}
\end{equation}
and we denote by $P^*$ its Legendre transform
\begin{equation}\label{PStarDef}
P^*(x) := \sup_{y \in \mathbb{R}^n}\left( y \cdot x - P(y)\right) = \sup_{y \in \Omega} (y \cdot x).
\end{equation}
The function $P^*$ is one-homogeneous, nonnegative, and convex, and the set
$$\Omega^* := \{P^* \leq 1\}$$ 
is the convex dual of $\Omega$. We note that $\Omega^* = \mathbb{R}^{n-d} \times \tilde{\Omega}^*$, where $\tilde{\Omega}^*$ (the $d$-dimensional convex dual of $\Omega$) is a compact convex polytope with the origin in its interior. 

We let $\Gamma_j^*$ denote the $j$-skeleton of $\Omega^*$ (note that $\Gamma_j^* = \emptyset$ when $j < n-d$),
and we let $S_j^* = \Gamma_j^* \backslash \Gamma_{j-1}^*$ denote the union of the interiors of the $j$-dimensional faces of $\Omega^*$. Finally, we let
$\Sigma_l = \emptyset$ for $l < n-d$, and we let
$$\Sigma_l := \begin{cases}
\mathbb{R}^{n-d} \times \{0\}, \quad l = n-d, \\
\{tx : t > 0 \text{ and } x \in S_{l-1}^*\}, \quad n-d < l \leq n
\end{cases}$$    
be the cone over $S_{l-1}^*$ in $\mathbb{R}^n$. We observe for $l \geq n-d$ that $\text{dim}(\Sigma_l) = l$, and that $P^*$ is linear when restricted to any connected component of $\Sigma_l$.

\begin{ex}
If $\Omega = \{0\}$, then $P^* = 0$ and $\Sigma_n = \mathbb{R}^n$.
\end{ex}

\begin{ex}
When $\Omega$ is the line segment connecting $\pm e_n$ we have 
$$P^* = |x_n|, \quad \Sigma_n = \{x_n > 0\} \cup \{x_n < 0\}, \quad \text{ and } \quad \Sigma_{n-1} = \{x_n = 0\}.$$
\end{ex}

\begin{ex}
When $\Omega$ is a regular tetrahedron in $\mathbb{R}^3$ centered at the origin, so is $\Omega^*$. In this case $\Sigma_1$ consists of the four open rays over the vertices of $\Omega^*$, $\Sigma_2$ consists of the six open planar sectors over the edges of $\Omega^*$, and finally $\Sigma_3$ consists of the four open solid sectors over the faces of $\Omega^*$.
\end{ex}

In this section we construct global solutions to an obstacle problem with $P^*$ as the obstacle. We first define what it means to be a global solution to the obstacle problem.

\begin{defn}
We say that a convex function $u^*: \mathbb{R}^n \rightarrow \mathbb{R}$ is a global solution to the obstacle problem with obstacle $P^*$ if $u^* \geq P^*$, $\det D^2u^* \leq 1$ in the Alexandrov sense, and $\det D^2u^* = 1$ in $\{u^* > P^*\}$. We let 
$$K := \{u^* = P^*\}$$ 
denote the contact set.
\end{defn}

\begin{ex}
The function $W_n$ defined by (\ref{FlatObstacle}) is a global solution to the obstacle problem with $P^* = 0$ and $K = B_1$.
\end{ex}

\noindent Below we say that $K$ has nonempty interior in a set $S$ if $S$ contains interior points of $K$. Our main result of this section is:

\begin{prop}\label{GlobalObstacleProblem}
There exists a global solution $u^*$ to the obstacle problem with obstacle $P^*$ such that $K$ is compact, has nonempty interior in each connected component of $\Sigma_k$ for $k > \frac{n}{2}$,
and $\Sigma_k \subset \mathbb{R}^n \backslash K$ for all $k \leq \frac{n}{2}$.
\end{prop}

\begin{ex}
When $\Omega$ is the line segment connecting $\pm e_n$, the contact set $K$ is the union of two compact convex sets with nonempty interior, one in $\{x_n \geq 0\}$ and the other in $\{x_n \leq 0\}$. When $n \leq 2$ these sets are disjoint and do not meet $\{x_n = 0\}$, but when $n \geq 3$ these sets meet along a convex set in $\{x_n = 0\}$ that contains interior points of $K$ (see Figure \ref{Fig1}).
\end{ex}


\begin{figure}
 \centering
    \includegraphics[scale=0.28, trim={0 35mm 0 10mm}]{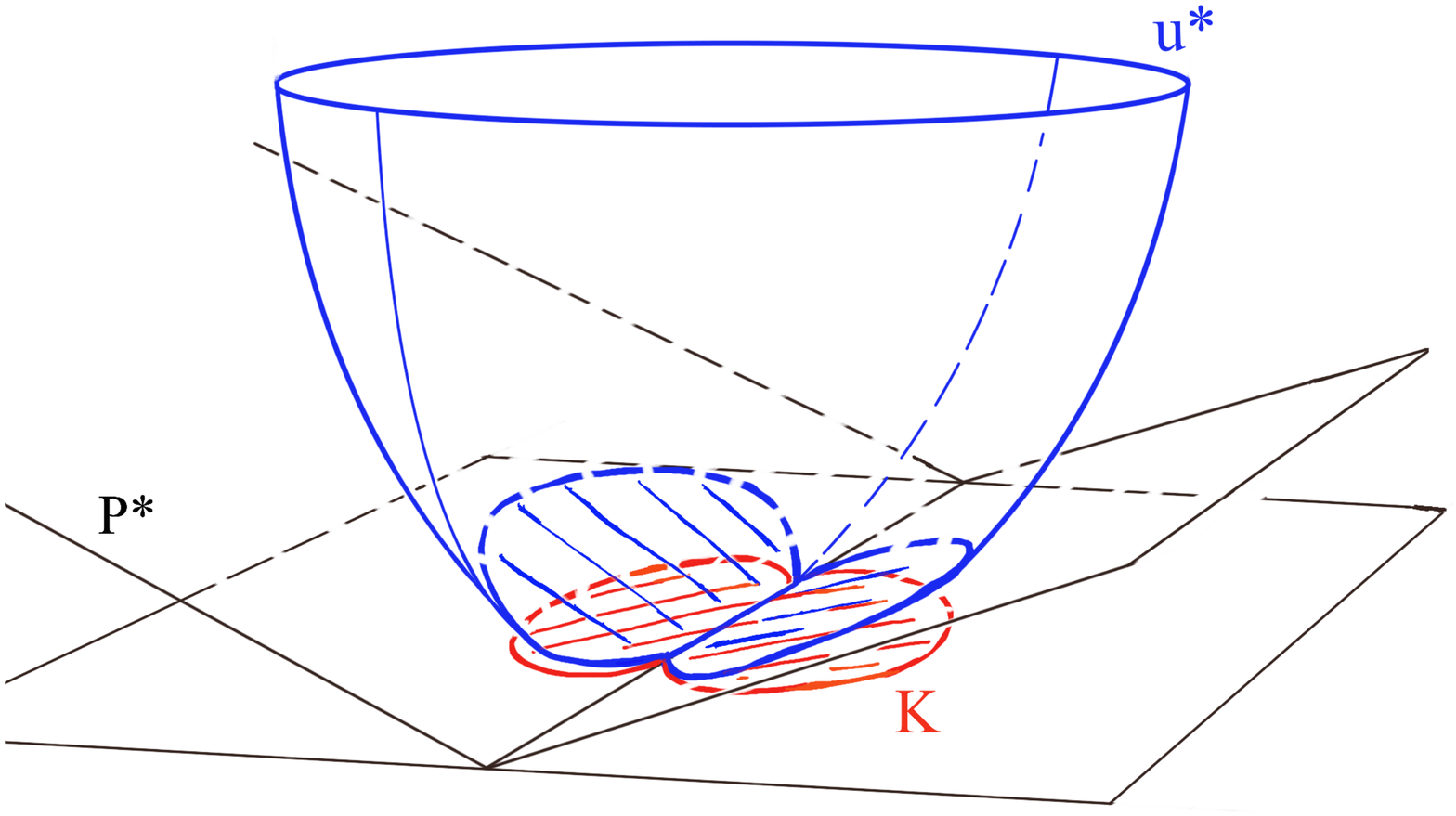}
 \caption{}
\label{Fig1}
\end{figure}

\begin{ex}
When $\Omega$ is a regular tetrahedron in $\mathbb{R}^3$ centered at the origin, the set $K$ is the union of four compact convex sets with nonempty interior, each intersecting one connected component of $\Sigma_3$. Each of these sets meets all three of the others along two-dimensional convex sets in the planar sectors that comprise $\Sigma_2$, but they do not intersect the origin or the rays that comprise $\Sigma_1$ (see Figure \ref{Fig2}).
\end{ex}

\begin{figure}
 \centering
    \includegraphics[scale=0.28, trim={0 25mm 0 10mm}]{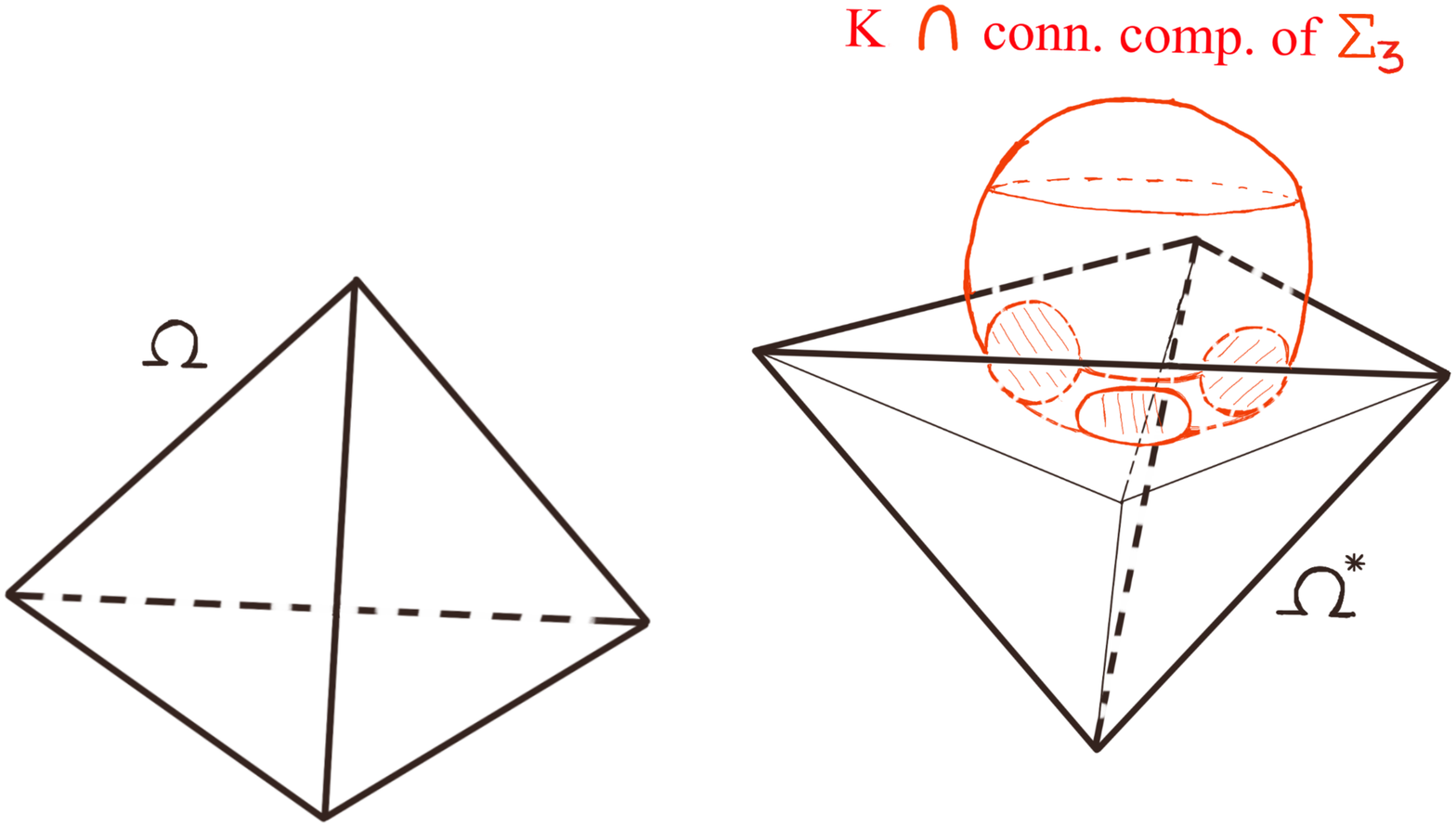}
 \caption{}
\label{Fig2}
\end{figure}

\noindent To conclude the section, we prove the proposition.

\begin{proof}[{\bf Proof of Proposition \ref{GlobalObstacleProblem}}]
It suffices to prove Proposition \ref{GlobalObstacleProblem} with the obstacle $\delta P^*$ for any fixed $\delta > 0$, since then $\delta^{-2}u^*(\delta x)$ is a global solution to the obstacle problem with obstacle $P^*$ and satisfies the remaining conditions.

Let
$$\varphi := W_n - 1 \text{ on } \mathbb{R}^n.$$
For $\epsilon > 0$ to be chosen, we can take $\delta$ small so that
$$W_n + \epsilon > \delta P^* \text{ on } \mathbb{R}^n.$$
Here and below we assume that $R >> 1$, and we denote by $u^*_R$ the solution to the obstacle problem in $B_R$ with obstacle $\delta P^*$, boundary data $\varphi$, and measure $\mu = dx$. Since $\det D^2 W_n \leq 1$, we have (see Remark \ref{ModifiedPerron}) that
$$u^*_R \leq W_n + \epsilon \text{ in } B_R.$$
Since $\varphi < 0 \leq \delta P^* \leq u^*_R$ in $B_1$ and $\det D^2 \varphi = 1$ outside of $B_1$, we have by the maximum principle that
$$\varphi \leq u^*_R \text{ in } B_R.$$
Up to taking a subsequence, the functions $u^*_R$ thus converge locally uniformly as $R \rightarrow \infty$ to a global solution $u^*$ to the obstacle problem with obstacle $\delta P^*$, which satisfies
\begin{equation}\label{GrowthGlobal}
\varphi \leq u^* \leq W_n + \epsilon \text{ on } \mathbb{R}^n.
\end{equation}
In particular,
\begin{equation}\label{FlatB1}
0 \leq \delta P^* \leq u^* \leq \epsilon \text{ in } B_1.
\end{equation}

We now examine $K = \{u^* = \delta P^*\}$. There exists $r_0 \in (0,\,1/2)$ depending only on $\Omega^*$ such that, for all $j \geq 1$, each (nonempty) connected component of $\Sigma_{j}$ contains a point $p \in \partial B_{1/2}$ such that $B_{r_0}(p)$ does not intersect any other connected component of $\Sigma_j$. Fix $n-k \geq 1$ and pick any such point $p \in \Sigma_{n-k}$. Then by the inequality (\ref{FlatB1}) the functions
$$Q(x) := r_0^{-2}\delta P^*(p + r_0\,x), \quad v(x) := r_0^{-2}u^*(p + r_0\,x)$$
satisfy
$$0 \leq Q \leq v \leq r_0^{-2}\epsilon \text{ in } B_1.$$ 
Furthermore, we may choose coordinates $x = (z,\,y)$ with $z \in \mathbb{R}^{n-k}$ and $y \in \mathbb{R}^k$, such that 
$$Q(x) = L(z) + M(y) \text{ in } B_1,$$ 
where $L$ is affine and $M(0) = 0$. Assume now that $k < \frac{n}{2}$, and let
$$w(x) := L(z) + w_{n,\,k}(x).$$
Since $w_{n,\,k}(x) \geq c(n,\,k)|y|$ (we take the right side to be $0$ when $k = 0$), we have for $\epsilon$ sufficiently small that 
$$w \geq Q \text{ in } B_1.$$ 
Furthermore, since 
$$w_{n,\,k} > c(n,\,k) > 0 \text{ on } \partial B_1,$$ 
we have for $\epsilon$ sufficiently small that 
$$w > \epsilon r_0^{-2} \geq v \text{ on } \partial B_1.$$ We conclude
that if $v(0) > Q(0) = w(0),$ then for some $h > 0$ the function $w + h$ touches $v$ from above somewhere in $\{v > Q\} \cap B_1$, which by (\ref{SuperIneq}) 
violates the maximum principle. By translating the origin a little and applying a similar argument e.g. in $B_{1/2}$,
we see that in fact $v = Q$ in a neighborhood of the origin, hence $u^* = \delta P^*$ in a neighborhood of $p$. This shows that $K$ has nonempty interior on each connected component
of $\Sigma_{n-k}$ when $k < \frac{n}{2}$.

Finally, assume by way of contradiction that $K$ contains a point $p \in \Sigma_{n-k}$ with $k \geq \frac{n}{2}$. Then 
$$\text{dim}(\partial u^*(p)) \geq \text{dim}(\partial P^*(p)) = k \geq \frac{n}{2},$$ 
so by Corollary \ref{PropagationCor}, $\det D^2 u^* \equiv 0$. This contradicts the inequality (\ref{GrowthGlobal}), which implies in particular that $u^*$ has a bounded sub-level set, in which it must have positive Monge-Amp\`{e}re mass by the Alexandrov maximum principle. The first inequality in (\ref{GrowthGlobal}) also implies that $K$ is compact, since $\varphi > \delta P^*$ outside a large ball, and this completes the proof.
\end{proof}

\section{Proof of Theorem \ref{Main}}\label{LowD}

In this section we let $u^*$ be a global solution to the obstacle problem with obstacle $P^*$, as obtained in Proposition \ref{GlobalObstacleProblem}. We will show in dimensions $n = 3$ and $4$ that the Legendre transform of $u^*$ satisfies the conditions of Theorem \ref{Main}. We first prove a more general result that holds in any dimension:

\begin{prop}\label{RegProp}
If $\partial u^* = \partial P^*$ on $K$, then the Legendre transform $u$ of $u^*$ satisfies the conditions of Conjecture \ref{nDMain}.
\end{prop}
\begin{proof}
We first claim that $u^*$ is smooth on $\mathbb{R}^n \backslash K$. If not, then by the results in \cite{CY} $u^*$ is affine along a line segment in $\mathbb{R}^n \backslash K$ that has an endpoint in $K$. At this endpoint, this contradicts that $\partial u^* = \partial P^*$. In particular, $\nabla u^*$ is a smooth measure-preserving diffeomorphism between $\mathbb{R}^n \backslash K$ and $\mathbb{R}^n \backslash \partial u^*(K)$. Since
$$\partial u^*(K) = \partial P^*(K) = \Gamma_{\left\lceil \frac{n}{2} - 1\right\rceil},$$
we conclude that $\nabla u$ is a smooth measure-preserving diffeomorphism between $\mathbb{R}^n \backslash \Gamma_{\left\lceil \frac{n}{2} - 1\right\rceil}$ and $\mathbb{R}^n \backslash K$. We also have on $\partial u^*(K) = \partial P^*(K) = \Gamma_{\left\lceil \frac{n}{2} - 1\right\rceil}$ that
$$u = u^{**} = P^{**} = P = 0.$$
It only remains to show that $\det D^2u = 1$ in the Alexandrov sense away from $\Gamma_0$, where it has Dirac masses. To see this, we note for $k < \frac{n}{2}$ that $\partial u(S_k)$ is contained in a finite union of codimension $k$ planes, hence it has measure zero for $k \geq 1$. Thus, for any Borel set $A \subset \mathbb{R}^n$ we have
\begin{align*}
|\partial u(A)| &= \left|\partial u\left(A \cap \Gamma_{\left\lceil \frac{n}{2} - 1\right\rceil}\right)\right| + \left|\nabla u\left(A \backslash \Gamma_{\left\lceil \frac{n}{2} - 1\right\rceil}\right)\right| \\
&= |\partial u (A \cap \Gamma_0)| + \left|A \backslash \Gamma_{\left\lceil \frac{n}{2} - 1\right\rceil}\right| \\
&= |\partial P (A \cap \Gamma_0) \cap K| + |A|,
\end{align*}
and the first term is positive if and only if $A \cap \Gamma_0$ is non-empty.
\end{proof}

We now specialize to dimension $n \leq 4$:

\begin{lem}\label{LowDRegProp}
Let $u^*$ be a global solution to the obstacle problem as obtained in Proposition \ref{GlobalObstacleProblem}, and assume $n \leq 4$. Then $\partial u^* = \partial P^*$ on $K$.
\end{lem}
\begin{proof}
When $n \leq 2$ each connected component $K_0$ of $K$ is compactly contained in a connected component of $\Sigma_n$, where $P^*$ is linear. The result thus reduces to the case $P^* = 0$, treated in \cite{S}. (Or, apply Proposition \ref{StrictConvexity} to conclude that $K_0$ is strictly convex, then apply Proposition \ref{ExtremalC1}).

So assume that $n = 3$ or $n = 4$. Then are just two cases to consider: points in $\partial K \cap \Sigma_n$, and in $\partial K \cap \Sigma_{n-1}$.
Consider first a point $y \in \partial K \cap \Sigma_{n-1}$. After a translation we may assume that $y = 0$, and after rotating, adding an affine function, and quadratically rescaling, we may assume that $u^*(0) = 0$, that 
$$u^* \geq (a x_n)_+ = P^* \text{ in } B_1$$
for some $a > 0$, and that the interior of $K$ intersects $B_1 \cap \{x_n = 0\}$. We need to show that $\partial u^*(0)$ is the segment $[0,\,ae_n]$ connecting the origin and $a e_n$. It is clear that 
$$[0,\,ae_n] \subset \partial u^*(0).$$ 
Recall that by the first inequality in (\ref{GrowthGlobal}), $u^*$ has quadratic growth. It follows from Corollary \ref{PropagationCor} and the fact that $n \leq 4$ that $\partial u^*(0)$ is contained in the $x_n$-axis. Finally, $\partial u^*(0)$ cannot contain a point of the form $be_n$ with $b > a$ or $b < 0$, otherwise $u^* > P^*$ in one of $\{x_n > 0\}$ or $\{x_n < 0\}$, contradicting that $K$ contains a ball centered on $B_1 \cap \{x_n = 0\}$. Thus $\partial u^*(0) = [0,\,ae_n]$.

It only remains to consider a point $y \in \partial K \cap \Sigma_n$. Let $K_0$ denote the intersection of $K$ with the connected component of $\Sigma_n$ containing $y$. After subtracting a linear function we may assume that $u^* = 0$ on $K_0$. If $y$ is an extremal point of $K_0$, then the conclusion follows from Proposition \ref{ExtremalC1}. By Proposition \ref{StrictConvexity}, the alternative is that $y$ is in the interior of a segment in $\partial K_0$ that has an endpoint in $\Sigma_{n-1}$. Normalize the picture as in the first case so that the endpoint of this segment which lies in $\Sigma_{n-1}$ is the origin, $u^*$ is tangent from above to $P^*$ at the origin and $P^* = (a x_n)_+$ for some $a > 0$ in a neighborhood of the origin, and $y \in \{x_n < 0\}$. Since $y$ is in the interior of the segment we have 
$$\partial u^*(y) \subset \partial u^*(0) = [0, ae_n].$$ 
If $be_n \in \partial u^*(y)$ with $b > 0$, it follows that that $u^*(0) > 0$, a contradiction. We conclude that 
$$\partial u^*(y) = \{0\}$$ 
is a single point, as desired.
\end{proof}

\noindent Theorem \ref{Main} (in fact, Conjecture \ref{nDMain} in dimensions $n \leq 4$) follows.

\begin{proof}[{\bf Proof of Theorem \ref{Main}}]
Let $u$ be the Legendre transform of the solution $u^*$ obtained in Proposition \ref{GlobalObstacleProblem}. When $n \leq 4$, the function $u$ satisfies the desired conditions by Lemma \ref{LowDRegProp} and Proposition \ref{RegProp}.
\end{proof}

\begin{rem} 
The simplest nontrivial instance of Theorem \ref{Main} is the case that $\Omega$ is the segment connecting $\pm e_n$ (see Figure \ref{Fig3}).
\end{rem}

\begin{figure}
 \centering
    \includegraphics[scale=0.28, trim={0 55mm 0 40mm}]{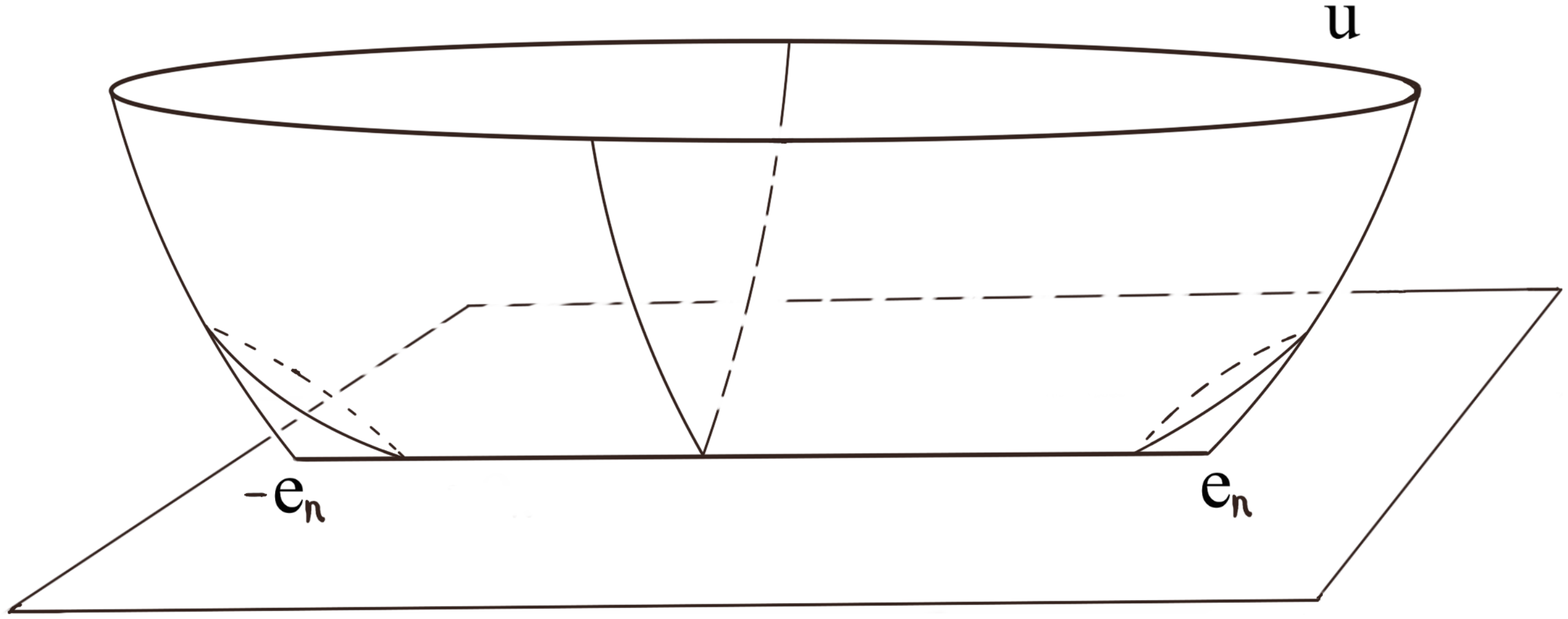}
 \caption{}
\label{Fig3}
\end{figure}

\begin{rem}
The key point when $n \leq 4$ is that $\partial u^*(p)$ is either a point or a line segment for any $p \in \mathbb{R}^n$ by the propagation result Corollary \ref{PropagationCor}. In contrast, if $n \geq 5$, then $\partial u^*(p)$ can be two-dimensional, so when one considers points e.g. in $\partial K \cap \Sigma_{n-1}$, then after normalizing as in the proof of Lemma \ref{LowDRegProp} so that $u^*$ is tangent from above to $(a x_n)_+$ at the origin, the propagation result doesn't prevent $\partial u^*(0)$ from having points off of the $x_n$-axis.
\end{rem}

\begin{rem}\label{5DCase}
A model case for Conjecture \ref{nDMain} in $\mathbb{R}^5$ would be to take coordinates $(z,\,y)$ with $z \in \mathbb{R}^3$ and $y \in \mathbb{R}^2$, and let $\Omega$ be a regular triangle in the $y$-plane centered at the origin. Then it is not hard to show that the function $u$ constructed as above satisfies that $u \geq 0$, and that $u$ is singular on $\{u = 0\} = \Omega$. However, it is not obvious to us whether $u$
is smooth away from $\Omega$. It remains a possibility, for example, that $u$ is affine when restricted to other triangles in the $y$-plane that share an edge with $\Omega$.
\end{rem}

\begin{rem}\label{Hessian}
By Theorem $1.5$ from \cite{JX}, $\Delta u$ is bounded by inverse distance from the singular set $\Gamma_1$ in our examples from Theorem \ref{Main}. In this remark we briefly discuss some model behaviors for the Monge-Amp\`{e}re metric $D^2u$ near $\Gamma_1$. 

First we discuss the edges. Since $\partial u$ maps each point in $S_1$ to a convex set of dimension $n-1$, the function $u$ has a Lipschitz singularity on each edge. In particular, $\Delta u$ grows exactly like inverse distance from $S_1$. A cylindrically symmetric model for such a singularity is the Pogorelov-type example
$$E(x) = \rho + \rho^{n/2}f(x_n),$$
where $x = (x',\,x_n),\, \rho = |x'|,$ and $f > 0$ is smooth, even and uniformly convex. (This example has Monge-Amp\`{e}re measure bounded between positive constants near the origin). Near the origin the metric $D^2E$ has the behavior
$$D^2E \sim \rho^{\frac{n}{2}-2}\, \nabla \rho \otimes \nabla \rho + \rho^{\frac{n}{2}}\, e_n \otimes e_n + \rho^{-1}\,(I - \nabla \rho \otimes \nabla \rho - e_n \otimes e_n).$$

We now discuss the vertices, where again $u$ has Lipschitz singularities. A model for the behavior of $u$ near an isolated vertex (the case $\Omega = \{0\}$) is the radial function
$$V(x) = r + r^{n+1},$$
where $r = |x|$. For $r$ small the metric $D^2V$ has the behavior
$$D^2V \sim r^{n-1}\, \nabla r \otimes \nabla r + r^{-1}(I - \nabla r \otimes \nabla r).$$
If the vertex is not isolated then $u$ is affine along each edge in $\Gamma_1$ that meets the vertex, and we expect that the behavior of the metric $D^2u$ transitions from the model $D^2V$ to the model $D^2E$ as one moves towards an edge.
\end{rem}

\section{Y-Shaped Singularities}\label{YSing}
The approach to constructing singular Monge-Amp\`{e}re metrics by solving an obstacle problem is flexible, and by changing the obstacle one can produce
examples similar to the ones from Section \ref{LowD} that are singular on a variety of graphs (not just the edges of convex
polytopes). For example, if we instead take $P^*$ to be the maximum of finitely many affine functions, then the approach produces examples that are affine when restricted to the line segments that connect pairs of points which are gradients of $P^*$ in open regions of $\mathbb{R}^n$ that share an $n-1$-dimensional face. 
To produce examples with singularities that form a Y-shape, we need to refine our choice of obstacle. In this section we indicate how to construct such examples, again in dimensions $n = 3$ and $n = 4$. More precisely, we show:

\begin{thm}\label{YShaped}
Let $\Gamma$ be a finite union of line segments in $\mathbb{R}^n$ that share a common vertex and point in distinct directions from this vertex. Assume further that $n = 3$ or $n = 4$. Then there exists a convex function $u: \mathbb{R}^n \rightarrow \mathbb{R}$ such that $u \in C^{\infty}(\mathbb{R}^n \backslash \Gamma)$, $u$ is affine when restricted to any of the segments in $\Gamma$, and
$$\det D^2u = 1 + \sum_{q \in \Gamma_0} a_q\delta_q,$$
where $\Gamma_0$ is the set of endpoints of the segments in $\Gamma$ and $a_q > 0$.
\end{thm}

\noindent In particular, when $\Gamma$ consists of three segments with a common endpoint we have a Y-shaped singular set. The singular set $\Gamma$ is not in general
contained in a level set of $u$ (unlike the examples in the previous section), but it is possible to make this happen when $\Gamma$ has certain symmetries (see Remark \ref{SingLevelSet}). Since the proof of Theorem \ref{YShaped} is similar to that of Theorem \ref{Main} we just sketch the main steps.

\begin{proof}[{\bf Proof of Theorem \ref{YShaped}:}]
After a translation we may assume that the common vertex is the origin. By taking quadratic rescalings, we see it suffices to prove Theorem \ref{YShaped} with singular set $\delta \Gamma$, for some $\delta > 0$ small to be chosen.

\vspace{1mm}

{\bf Step 1: Obstacle Problem.} The first step is to solve a global obstacle problem. In this step the dimension $n$ is arbitrary. For some $r_0 \in (0,\,1/4)$ depending on the directions of the segments in $\Gamma$, we can find a collection $\{L_i\}_{i = 1}^M$ of affine functions such that $\nabla L_i$ are the nonzero endpoints of the segments in $\delta\Gamma$, and the sets $\overline{B_1} \cap \{L_i \geq 0\}$ are congruent, pairwise disjoint, and have exterior tangent ball $B_{1-r_0}$ at points $p_i$. Take $\epsilon > 0$ small so that $\{W_n \leq \epsilon\} \cap \{L_i \geq 0\}$ are also pairwise disjoint, then take $\delta$ small so that $\{L_i \geq (W_n - \epsilon)_+\}$ are pairwise disjoint and
$$\max_{i \leq M}\, L_i < W_n + \tilde{\epsilon} \text{ on } \mathbb{R}^n,$$
for some $\tilde{\epsilon} > 0$ to be chosen later. Finally, let
$$\varphi(x) := \max\{W_n-\epsilon,\,0,\,L_1,\,...,\,L_M\}.$$
We observe that the sets $\{\varphi = L_i\} = \{L_i \geq (W_n-\epsilon)_+\}$ are pairwise disjoint.

For $R >> 1$, let $u^*_R$ be the solution to the obstacle problem in $B_R$ with boundary data and obstacle equal to $\varphi$, and measure $dx$. Since $W_n + \tilde{\epsilon}$ is a supersolution to the equation that lies above the obstacle, we have that
$$\varphi \leq u^*_R \leq W_n + \tilde{\epsilon} \text{ in } B_R.$$
Up to taking a subsequence, the functions $u^*_R$ thus converge locally uniformly as $R \rightarrow \infty$ to a function $u^*$ on $\mathbb{R}^n$ that satisfies
$$0 \leq \varphi \leq u^* \leq W_n + \tilde{\epsilon} \text{ on } \mathbb{R}^n, \quad \det D^2u^* \leq 1, \quad \det D^2u^* = 1 \text{ in } \{u^* > \varphi\}.$$

\vspace{1mm}

{\bf Step 2: The Contact Set.} The second step is to study the geometry of
$$K := \{u^* = \varphi\},$$
and to show that to show that $\partial \varphi(K) = \delta \Gamma$. In this step we assume that $n \geq 3$.
We claim that $K$ has nonempty interior in each of the (pairwise disjoint) $n-1$-dimensional balls $B_1 \cap \{L_i = 0\}$, and that $K \subset \{\varphi > W_n - \epsilon\}$.
To prove the first claim we use that $0 \leq u^* \leq \tilde{\epsilon}$ in $B_1$. Provided $\tilde{\epsilon}$ is sufficiently small depending on $r_0$, we can use the barrier $w_{n,\,1}$ in the same way as in the proof of Proposition \ref{GlobalObstacleProblem} to show that $K$ contains a neighborhood of each point $p_i$.
The second claim follows from the strong maximum principle. If $u^* = \varphi$ at a point
in $\{\varphi = W_n - \epsilon\}$, then $u^* + \epsilon$ touches $W_n$ from above at some point in the open set $\{W_n > \epsilon/2\}$. However, $u^* + \epsilon > W_n$
in a neighborhood of $\{W_n = \epsilon/2\}$. Since $W_n$ smoothly solves $\det D^2W_n = 1$
in $\{W_n > \epsilon/2\}$, this contradicts the strong maximum principle.

The set $K$ is thus the union of $M+1$ compact convex sets, one ``central" set (contained in $\{\varphi = 0\}$) and $M$ ``external" sets (each contained in one of $\{\varphi = L_i\}_{i = 1}^M$), such that the central set meets each external set along an $n-1$-dimensional face and the external sets are pairwise disjoint (see Figure \ref{Fig4} for the case $M = 3$). Furthermore, since $K \subset \{\varphi > W_n-\epsilon\}$, we see that $u^*$ is in fact a global solution to the obstacle problem with obstacle 
$$P^* := \max\{0,\, L_1,\, ... ,\,L_M\},$$
and that
$$\partial \varphi(K) = \partial P^*(K) = \delta\Gamma.$$

\vspace{1mm}

{\bf Step 3: Subgradients on the Contact Set.} The last step is to show when $n = 3$ or $4$ that the Legendre transform of $u^*$ satisfies the conditions of Theorem \ref{YShaped}. It suffices to show that $\partial u^* = \partial P^*$ on $K$, by essentially the same argument as in the proof of Proposition \ref{RegProp}. (The only difference in this case is that the Legendre transform $u$ is affine, rather than zero, on each segment in $\delta\Gamma$.) Showing that $\partial u^* = \partial P^*$ on $K$ when $n = 3$ or $4$ is the same as in the proof of Lemma \ref{LowDRegProp}, and this concludes the proof.
\end{proof}

\begin{figure}
 \centering
    \includegraphics[scale=0.28, trim={0 25mm 0 10mm}]{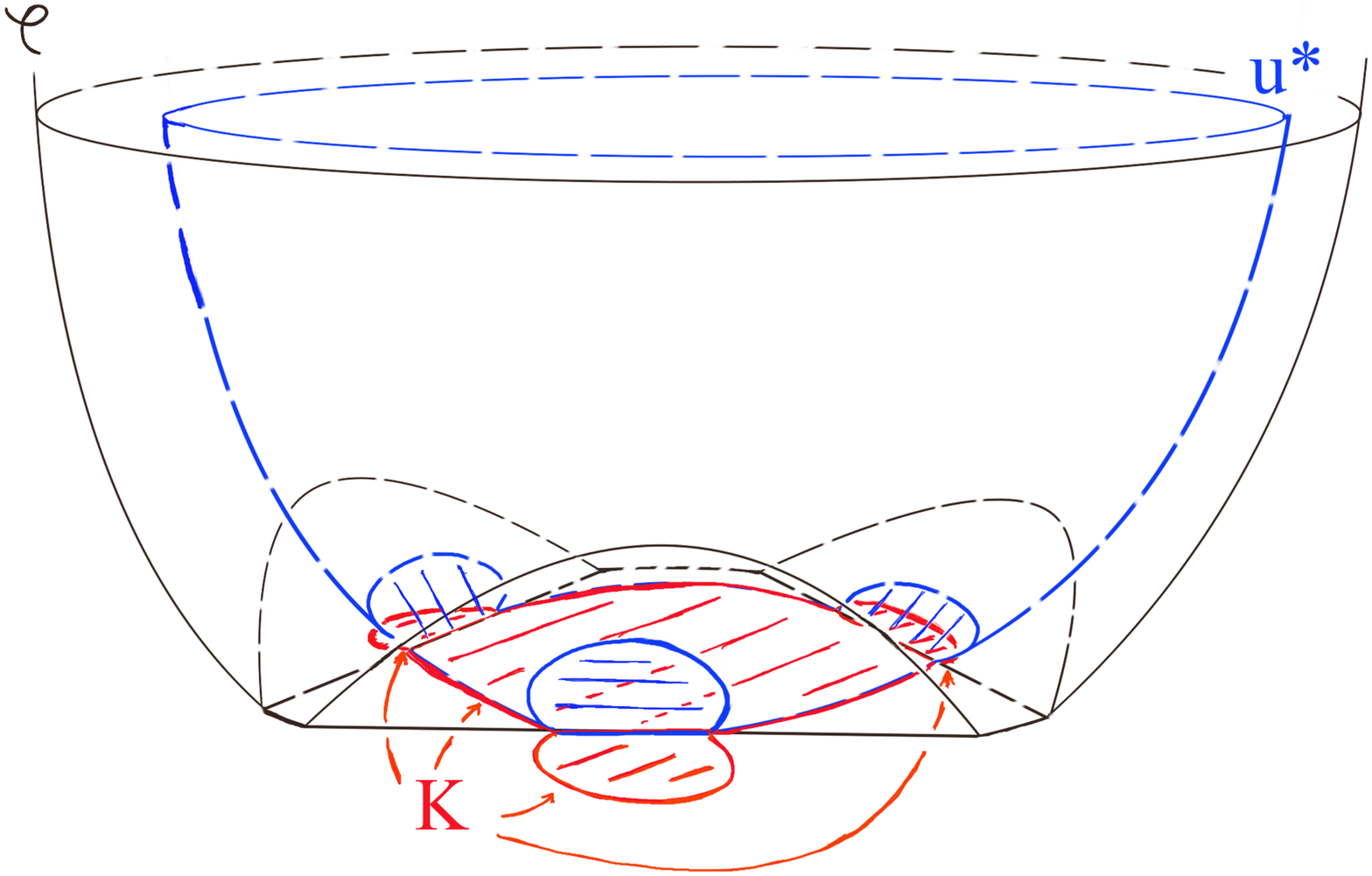}
 \caption{}
\label{Fig4}
\end{figure}

\begin{rem}
The idea is to slowly lower the boundary data for the obstacle problem with obstacle $P^*$. At first the solution will stick to the obstacle on only one region where $P^*$ is affine. Eventually, it will stick on all such regions and the $n-1$-dimensional faces that join them. The game is to stop somewhere in between.
\end{rem}

\begin{rem}\label{SingLevelSet}
It is possible to construct instances of Theorem \ref{YShaped} where $\Gamma$ lies in a set where $u$ is linear, for example when $n = 3$ and $\Gamma$ consists of segments of unit length that start at the origin and end on the vertices of a regular polygon that does not contain the origin.
\end{rem}



\end{document}